\documentclass[11pt]{article}

\usepackage[a4paper,margin=1in]{geometry}
\usepackage{tikz}
\usepackage{float}
\usepackage[T1]{fontenc}
\usepackage[utf8]{inputenc}
\usepackage{booktabs}
\usepackage{lmodern}
\usepackage{microtype}
\usepackage{amsmath,amssymb,amsthm,mathtools}
\usepackage{enumitem}
\usepackage[colorlinks=true,linkcolor=blue,citecolor=blue,urlcolor=blue]{hyperref}

\title{Cycle lengths and chords under chromatic and degree constraints}
\author{Xiaozheng Chen\footnote{School of Mathematics and Statistics, Zhengzhou University, Zhengzhou 450000, P.R. China. Email: cxz@zzu.edu.cn. Supported by National Nature Science Foundation of China (No. 12301457).}~~~ and Bo Ning\footnote{College of Computer Science, Nankai University, Tianjin 300350, P.R. China.
E-mail: \texttt{bo.ning@nankai.edu.cn} (B. Ning). Partially supported by the National Nature Science
Foundation of China (No. 12371350).}}
\date{}

\theoremstyle{plain}
\newtheorem{theorem}{Theorem}[section]
\newtheorem{lemma}[theorem]{Lemma}
\newtheorem{proposition}[theorem]{Proposition}
\newtheorem{corollary}[theorem]{Corollary}
\newtheorem{fact}[theorem]{Fact}

\theoremstyle{definition}
\newtheorem{cons}[theorem]{Construction}

\newtheorem{problem}[theorem]{Problem}
\newtheorem{conjecture}[theorem]{Conjecture}
\newtheorem{remark}[theorem]{Remark}
\newtheorem{claim}{Claim}[section]

\newcommand{\calL}{\mathcal L}
\newcommand{\calP}{\mathcal P}
\newcommand{\calD}{\mathcal D}
\newcommand{\abs}[1]{\lvert #1\rvert}

\begin{document}
\maketitle

\begin{abstract}
We mainly consider three problems on cycle lengths and cycles with chords in graphs:
\begin{enumerate}
\item[(a)] Gao, Huo, and Ma \cite[Question~1.5]{GaoHuoMa2021} asked whether, for every fixed $k\ge3$, there is a function $f_k(n)\to\infty$ such that every $n$-vertex $(k+1)$-critical graph contains $f_k(n)$ consecutive cycle lengths.
\item[(b)] Let $g_k(n)$ be the maximum integer $t$ such that every $n$-vertex $k$-critical graph with $k\ge4$ contains an odd cycle with at least $t$ chords. Voss conjectured (see \cite[pp.~168]{VossBook}) that $g_k(n)\to\infty$ as $n\to\infty$ for each $k\ge4$, which extends a 1976 conjecture of Erd\H{o}s (see also Erd\H{o}s Problem~1091 \cite{Bloom1091}).
\item[(c)] K\'ara and Kr\'al \cite{KaraKral2003} conjectured that every graph on $31$ vertices with minimum degree at least $8$ contains a cycle with at least $31$ chords.
\end{enumerate}
We answer question (a) in the negative for $k=3$, and disprove conjecture (b) for all $k\ge5$. We point out the work of Alexeev-Putterman-Sawhney-Sellke-Valiant (2026) on Erd\H{o}s Problem 1901 disproves the case $k=4$ for conjecture (b).
We prove conjecture (c). We also discuss two other related problems in the part of concluding remark.
\end{abstract}



\section{Introduction}

All graphs in this paper are finite and simple. The length of a cycle is the number of its edges. For a graph $G$, let $\chi(G)$ denote its
chromatic number. We denote $\calL(G)=\{\ell:\text{$G$ contains a cycle of length $\ell$}\}$, $\calL_e(G)=\{2\ell:\text{$G$ contains a cycle of length $2\ell$}\}$, and
$\calL_o(G)=\{2\ell+1:\text{$G$ contains a cycle of length $2\ell+1$}\}$.
A graph is $k$-critical if $\chi(G)=k$ and every proper subgraph of $G$ is
$(k-1)$-colorable. In this paper,
we study how the chromatic number, criticality, connectivity, and minimum-degree assumptions force arithmetic structure in cycle lengths and force cycles with many chords.

One elementary consequence of large chromatic number for a graph is the existence of a long
cycle: every graph of chromatic number $\chi\ge 3$ contains a cycle of length at
least $\chi$. Instead of
asking for only one long cycle, one can ask how much arithmetic structure must be
present in the cycle-length set $\mathcal L(G)$, and how this structure changes
when the graph is assumed to be $k$-critical.

The first kinds of results concerns the relationship between chromatic number and the lengths of cycles of a prescribed parity. Bollob\'as and Erd\H{o}s
(see \cite{Erdos1990}) conjectured that 
$\chi(G)\leq 2|\mathcal L_o(G)|+2$; this would be best possible, consider
$G=K_{2k+2}$, where $|\mathcal L_o(G)|=k$. Gallai (see \cite{Gyarfas1992}) proposed the stronger
conjecture that if $G$ is 2-connected, $|\mathcal L_o(G)|=k$, and
$G\not\cong K_{2k+2}$, then $\delta(G)\leq 2k$. Both conjectures were
eventually proved by Gy\'arf\'as \cite{Gyarfas1992}. Mih\'ok and
Schiermeyer \cite{MihokSchiermeyer2004} established the corresponding
bound for even cycles, namely $\chi(G)\leq 2|\mathcal{L}_e(G)|+3.$
Consequently, if $\chi(G)=k+1$, then $G$ contains cycles of at least $\lfloor k/2\rfloor$ distinct odd lengths and at least
$\lfloor k/2\rfloor-1$ distinct even lengths. Thus, large chromatic
number compels the presence of many cycle lengths in each parity class,
although it gives no control over how tightly those lengths are spaced.

A complementary direction seeks intervals within the set of cycle
lengths. Kostochka, Sudakov, and Verstra\"ete
\cite{KostochkaSudakovVerstraete2017} proved, in a stronger form, a
conjecture of Erd\H{o}s \cite{Erdos1992}: for sufficiently large $k$, every triangle-free graph $G$ with
$\chi(G)=k$ contains cycles of $\Omega(k^2\log k)$ consecutive lengths. For arbitrary graphs, Gao, Huo, Liu, and Ma
(see \cite[Theorem 4.5]{GaoHuoLiuMa2022}) proved that every graph with
$\chi(G)\geq k+2$ contains cycles of $k$ consecutive lengths. This
settles a conjecture of Sudakov and Verstra\"ete \cite{SudakovVerstraete2017}, and the bound is sharp
for complete graphs.

Gao, Huo, and Ma \cite[Theorem~1.3]{GaoHuoMa2021} strengthened the theorem of Gy\'arf\'as \cite{Gyarfas1992}. As a consequence, they proved that every graph $G$ with
$\chi(G)=k+1$ contains $\lfloor k/2\rfloor$ cycles whose lengths are
consecutive odd lengths, that is, lengths of the form
$2m+1,2m+3,\ldots,2m+2\lfloor k/2\rfloor-1$. They \cite[Theorem~1.2]{GaoHuoMa2021} also obtained a stronger
interval theorem for large $k$: if $k\ge 6$ and $\chi(G)=k+1$, then $G$
contains cycles whose lengths are $k$ consecutive integers, unless some block
of $G$ is isomorphic to $K_{k+1}$. As a
consequence, every noncomplete $(k+1)$-critical graph with $k\ge 6$ contains
cycles of all lengths modulo $k$  (see \cite[Theorem~1.4]{GaoHuoMa2021}).
These results are close to best possible in terms of chromatic number alone. For
$k\ge 3$, Gao, Huo, and Ma \cite{GaoHuoMa2021}  observed that the join $K_{k-2}\vee C_5$ is
$(k+1)$-critical and has cycle-length set exactly
$\{3,4,\ldots,k+3\}$. Thus, the chromatic number alone cannot force intervals of cycle lengths that grow with the order of the graph.  However, it remains possible that criticality together with large order forces such a conclusion. This motivates the following question.

\begin{problem}[Gao--Huo--Ma \rm {\cite[Question~1.5]{GaoHuoMa2021}}]
\label{prob:ghm}
Let $k\ge 3$. Does there exist a function $f_k(n)$, with
$f_k(n)\to\infty$ as $n\to\infty$, such that every $n$-vertex
$(k+1)$-critical graph contains cycles whose lengths are $f_k(n)$
consecutive integers?
\end{problem}

We first give a negative answer to Problem~\ref{prob:ghm}
for $k=3$. 

\begin{theorem}\label{thm:ghm-false}
Problem \ref{prob:ghm} is false for $k=3$.
\end{theorem}

Let $C$ be a cycle in a graph $G$.
A \textit{chord} of a cycle $C$ is an edge of $G$ joining two vertices of $C$ that
are not consecutive on $C$.
The number of chords of
$C$ in $G$ is denoted by $\operatorname{ch}_G(C)$.
The next conjecture concerns cycles with a prescribed number of chords in
$k$-critical graphs and was proposed by Voss in 1991 (see
Section~8.5 of \cite{VossBook}).

\begin{conjecture}[Voss \rm {\cite[pp.~168]{VossBook}}]\label{prob:voss}
Let $g_k(n)$ be the largest integer $\alpha$ such that every $n$-vertex
$k$-critical graph with $k\ge 4$ contains an odd cycle with at least
$\alpha$ chords. Is it true that $g_k(n)\to\infty$ as $n\to\infty$?
\end{conjecture}

This conjecture is related to an old conjecture of Erd\H{o}s. 
Erd\H{o}s \cite[pp.~4]{Erdos1976} asked the following question on odd cycle with chords in $K_4$-free graphs in 1976; it is now listed as Erd\H{o}s Problem~1091
\cite{Bloom1091}.

\begin{problem}[Erd\H{o}s, Problem~1091]\label{prob:erdos1091}
Let $G$ be a $K_4$-free graph with $\chi(G)=4$. Must $G$ contain an odd
cycle with at least two chords? More generally, does there exist a function
$f(r)\to\infty$ such that every graph $G$ with $\chi(G)=4$ and with every
subgraph on at most $r$ vertices $3$-colorable contains an odd cycle with at
least $f(r)$ chords?
\end{problem}

The first part of Problem~\ref{prob:erdos1091} has a positive answer.
Larson~\cite[Theorem~1]{Larson1979} showed that a $K_4$-free graph
with chromatic number $4$ contains an odd cycle with a chord.
The stronger conclusion asked for in the problem was proved by
Voss~\cite[Theorem~2]{Voss1982}: such a graph contains an odd cycle with at
least two chords.
The second part of Problem~\ref{prob:erdos1091}, however, has a negative
answer.  Alexeev, Putterman, Sawhney, Sellke, and Valiant
\cite[Theorem 4.1]{APSSV2026} disproved it by giving a counterexample.

\begin{theorem}[Alexeev--Putterman--Sawhney--Sellke--Valiant \rm {\cite[Theorem 4.1]{APSSV2026}}]\label{thm:apssv}
For every integer $m\ge 1$, there exists a $K_4$-free graph $G_m$ on
$20m+31$ vertices such that $\chi(G_m)=4$, every proper subgraph of $G_m$ is
$3$-colorable, and every cycle of $G_m$ has at most ten chords.
\end{theorem}

This disproves Conjecture~\ref{prob:voss} when $k=4$.
Note that the graphs $G_m$ are $4$-critical, their orders tend to infinity,
and every odd cycle in $G_m$ has at most ten chords. In particular,
$g_4(20m+31)\le 10$ for all $m\ge 1$.

Motivated by the construction of Alexeev et al., we disprove Conjecture~\ref{prob:voss} for every fixed $k\ge 5$ by constructing a counterexample $H_{k,m}$, which is described in Section~\ref{sec:voss}. Our main result is as follows.

\begin{theorem}\label{thm:voss-fixed-intro}
For  every fixed $k\ge4$, there are infinitely many $k$-critical graphs in which every cycle has at most a constant number of chords, depending only on $k$.
More precisely, for every $m\ge 1$ the graph $H_{k,m}$ is $k$-critical, has $(25k-80)m+26k-73$ vertices, and each of its cycles has at most
$\binom{k-3}{2}+10(k-3)^2+60(k-3)$
chords.
\end{theorem}
Alexeev et al.'s construction already disproves Voss's conjecture for the case $k=4$. In fact, after natural relabeling, our graph $H_{4,m}$
is isomorphic to their graph $G_m$
$m$. The new contribution of the present construction is its extension to every $k\geq 5$.

Several results show that large minimum degree forces cycles with many chords; we recall those needed in the following.
Gupta, Kahn, and Robertson \cite{GuptaKahnRobertson1980} proved lower bounds in terms
of minimum degree. 

\begin{theorem}[Gupta-Kahn-Robertson \rm \cite{GuptaKahnRobertson1980}]\label{thm:GuptaKahnRobertson}
Let $G$ be a graph with minimum degree at least $d$, where $d\ge 2$.  Then $G$ contains a cycle with at least $(d+1)(d-2)/2$ chords.
\end{theorem}

Ash~\cite{Ash1985} later proved the following refined lower bound for $2$-connected
graphs of order at least $2d$.

\begin{theorem}[Ash~\cite{Ash1985}]\label{thm:Ash1985}
Let $G$ be a $2$-connected graph with minimum degree at least $d$, where $d\ge 3$.
If $|V(G)| \ge 2d,$
then $G$ contains a cycle with at least $d(d-2)$ chords.
\end{theorem}
Tian and Zang  \cite{TianZang1991} strengthened Ash's bound by one when the order is at least
$2d+1$, apart from the exceptional graphs.
\begin{theorem}[Tian--Zang \cite{TianZang1991}]\label{thm:tian-zang}
Let $G$ be a $2$-connected graph with minimum degree at least $d$, where $d\ge 3$.
If $|V(G)| \ge 2d+1$, then $G$ contains a cycle with at least  $d(d-2)+1$ chords, unless $G$ is $K_{d,m},\ m>d$ or the Petersen graph.
\end{theorem}

Let $f(n,c)$ denote the least integer $d$ such that every $n$-vertex graph with minimum degree at least $d$ contains a cycle with at least $c$ chords. 
K\'ara and Kr\'al \cite{KaraKral2003} proved $\sqrt{c/2}\le f(n,c)\le3\sqrt{c}$ and 
$8\leq f(31,31)\leq 9$.
Moreover, they
conjectured that the equality should be $8$. 

\begin{conjecture}[K\'ara and Kr\'al \cite{KaraKral2003}]
$f(31,31)=8$.    
\end{conjecture}

We first extend Ash's theorem in the following.
\begin{theorem}
\label{thm:one-exception-ash}
Let $d\ge 3$.  Let $G$ be a 2-connected graph and let $z\in V(G)$.
Suppose that $|V(G)|\geq 2d+1$ 
and $d_G(v)\ge d$ for every $v\in V(G)\setminus\{z\}.$ Then $G$ contains a cycle with at least
$d(d-2)$ chords.
\end{theorem}

With the above theorem as a main tool, we establish the conjecture of K\'ara–Kr\'al.

\begin{theorem}\label{thm:KaraKral-conj}
$f(31,31)=8$. 
\end{theorem}

Finally, we mention an application of the theorems of Gupta-Kahn-Robertson \cite{GuptaKahnRobertson1980} and Tian–Zang \cite{TianZang1991}. The first part of Erd\H{o}s Problem~1091 is best possible with respect to
the number of chords in an odd cycle.
Indeed, the
pentagonal wheel is $K_4$-free and $4$-chromatic, but every odd cycle in
it has at most two chords. 
Voss \cite[Conjecture~8.5.5, pp. 169]{VossBook} proposed a related conjecture in which the cycle is not
required to be odd: every $K_4$-free graph $G$ with $\chi(G)\ge4$
contains a cycle with at least four chords.
We establish it.

\begin{proposition}\label{prop:four-chord-main}
Let $G$ be a  graph with $\chi(G)\ge 4$.  If $G$ contains no subgraph isomorphic to $K_4$, then $G$ contains a cycle with at least four chords.
\end{proposition}

The organization of this paper is as follows. In Section \ref{sec:ghm}, we disprove the case $k=3$ of the Gao-Huo-Ma problem. In Section \ref{sec:voss}, we disprove Conjecture~\ref{prob:voss} for every $k\ge 5$. In Section 4, we prove the conjecture of K\'ara--Kr\'al, by proving $f(31,31)=8$. We conclude this paper by considering two related problems in the last section.

\section{The case \texorpdfstring{$k=3$}{k=3} of the Gao--Huo--Ma problem}\label{sec:ghm}
We now disprove Problem~\ref{prob:ghm} for $k=3$. Recall that $\calL(G)$ denotes the set of cycle lengths in $G$.  Let $\rho(G)$ be the largest integer $m$ for which some interval of $m$ consecutive integers is contained in $\calL(G)$.

We use the following classical property of Haj\'os's construction.  Given vertex-disjoint graphs $X$ and $Y$ and edges $x_1x_2\in E(X)$ and $y_1y_2\in E(Y)$, the Haj\'os join along these edges is obtained by deleting $x_1x_2$ and $y_1y_2$, identifying $x_1$ with $y_1$, and adding the edge $x_2y_2$.
The following is the standard preservation theorem for Haj\'os's construction; see Haj\'os \cite{Hajos1961}, Jensen and Toft \cite[Chapter 11]{JensenToft1995}, or Jensen and Royle \cite{JensenRoyle1999}.

\begin{fact}[Haj\'os join and criticality]\label{fact:hajos}
For $r\ge 3$, the Haj\'os join of two $r$-critical graphs is $r$-critical.  
\end{fact}

We define a sequence of rooted graphs. For each $t\ge 1$, the graph $R_t$ has ordered terminals $(a_t,b_t)$, and $G_t$ is obtained from $R_t$ by adding the terminal edge $a_tb_t$. Start with $G_1=K_4$, choose an arbitrary edge $a_1b_1$ in $G_1$, and set $R_1=G_1-a_1b_1$.
Suppose that $R_t$, $G_t$, and the terminals $(a_t,b_t)$ have been defined. We construct $G_{t+1}$ from $G_t$ by a Haj\'os join with a new copy of $K_4$. Let the vertex set of this new $K_4$ be $\{x,a_{t+1},b_{t+1},z_{t+1}\}$. Choose the edge $a_tb_t$ in $G_t$ and the edge $x z_{t+1}$ in the new $K_4$, delete them, identify $x$ with $a_t$, and add the edge $b_tz_{t+1}$. The resulting graph is $G_{t+1}$, and its new terminal edge is $a_{t+1}b_{t+1}$ from the added $K_4$. Finally, set
$R_{t+1}=G_{t+1}-a_{t+1}b_{t+1}.$

\begin{proposition}\label{prop:hajos-critical-size}
For every $t\ge 1$, the graph $G_t$ is $4$-critical and has $3t+1$ vertices.
\end{proposition}

\begin{proof}
We proceed by induction on $t$. For $t=1$, $G_1=K_4$ is $4$-critical and has $4$ vertices.
Assume that $G_t$ is $4$-critical. By the construction, $G_{t+1}$ is obtained from $G_t$ and a new copy of $K_4$ via a Haj\'os join. Since $K_4$ is $4$-critical, by Fact~\ref{fact:hajos}, $G_{t+1}$ is also $4$-critical.
Since $G_1$ has $4$ vertices and each step adds exactly three new vertices $a_{t+1},b_{t+1},z_{t+1}$,  $|V(G_t)| = 4 + 3(t-1) = 3t+1$.
\end{proof}

Let $H_{t+1}$ be the subgraph of $R_{t+1}$ with vertex set $\{a_t,b_t,a_{t+1},b_{t+1},z_{t+1}\}$ and edge set consisting of the five new edges
$a_{t+1}a_t,\; b_{t+1}a_t,\; a_{t+1}z_{t+1},\; b_{t+1}z_{t+1},\; z_{t+1}b_t.$
Then $H_{t+1}$ meets $R_t$ exactly in the two vertices $a_t$ and $b_t$.

For $t\ge 1$, define $\calP_t$ to be the set of lengths of $a_tb_t$-paths in $R_t$, and define $\calD_t$ to be the set of cycle lengths in $R_t$.  If $S\subseteq \mathbb Z$ and $c$ is an integer, write $S+c=\{s+c:s\in S\}$.

\begin{lemma}\label{lem:recurrences}
For every $t\ge 1$, one has $\calP_{t+1}=\{2\}\cup(\calP_t+3)$ and $\calD_{t+1}=\calD_t\cup\{4\}\cup(\calP_t+3)$.  Moreover, $\calL(G_t)=\calD_t\cup(\calP_t+1)$.
\end{lemma}

\begin{proof}
The edges of $R_{t+1}$ not already in $R_t$ are
$a_{t+1}a_t$, $b_{t+1}a_t$, $a_{t+1}z_{t+1}$, $b_{t+1}z_{t+1}$, $z_{t+1}b_t$.
Thus, $R_t$ attaches to the new vertices only at $a_t$ and $b_t$; at $a_t$ the external edges are $a_t a_{t+1}$ and $a_t b_{t+1}$, while at $b_t$ the only external edge is $b_t z_{t+1}$.

Let $P$ be an $a_{t+1}b_{t+1}$-path in $R_{t+1}$.
If $P\subseteq H_{t+1}$, then $P$ is either $a_{t+1}a_t b_{t+1}$ or $a_{t+1}z_{t+1}b_{t+1}$, both of length $2$.
If $P\cap R_t\neq \emptyset$, then, because $P$ is simple and the only attachment vertices of $R_t$ are $a_t$ and $b_t$, the intersection $P\cap R_t$ is a single $a_tb_t$-path.
Hence $P$ has the form $a_{t+1}a_t R b_t z_{t+1}b_{t+1}$ or $a_{t+1}z_{t+1}b_t R a_t b_{t+1}$, where $R$ is an $a_tb_t$-path in $R_t$, and so $|P| = |R| + 3$.
Consequently, $\mathcal{P}_{t+1} = \{2\} \cup (\mathcal{P}_t+3)$.

Now let $C$ be a cycle in $R_{t+1}$.
If $C\subseteq R_t$, then $|C|\in\mathcal{D}_t$.
If $C\subseteq H_{t+1}$, since in $H_{t+1}$, the vertex $b_t$ has degree one, the only such cycle is $a_t a_{t+1}z_{t+1}b_{t+1}a_t$, of length $4$.
It remains to consider the case where $C$ uses edges from both $R_t$ and $H_{t+1}$.  Since $R_t$ and $H_{t+1}$ share exactly the two vertices $a_t$ and $b_t$, the cycle $C$ must consist of an $a_tb_t$-path in $R_t$ and an $a_tb_t$-path in $H_{t+1}$.  The only $a_tb_t$-paths in $H_{t+1}$ are
$a_t a_{t+1} z_{t+1} b_t$
and
$a_t b_{t+1} z_{t+1} b_t,$
both of length $3$.  Hence, $C$ has length $|R|+3$ for some $a_tb_t$-path $R$ in $R_t$, so $|C|\in\mathcal P_t+3$. 
Conversely, $\calD_{t+1}=\calD_t\cup\{4\}\cup(\calP_t+3)$.

Finally, $G_t = R_t + a_tb_t$.
A cycle in $G_t$ either avoids the edge $a_tb_t$ and thus lies in $R_t$, or uses it and becomes an $a_tb_t$-path in $R_t$ upon removal.
Hence, the set of cycle lengths of $G_t$ is $\mathcal{D}_t \cup (\mathcal{P}_t+1)$.
\end{proof}

\begin{lemma}\label{lem:explicit-pd}
For every $t\ge 1$, $\calP_t=\{2+3i:0\le i\le t-1\}\cup\{3t\}$ and $\calD_t=\{3,4\}\cup\{5+3i,6+3i:0\le i\le t-2\}$, where the second set in the formula for $\calD_t$ is empty for $t=1$.
\end{lemma}

\begin{proof}
We proceed by induction on $t$. For $t=1$, the graph $R_1=K_4-a_1b_1$ has $a_1b_1$-paths of lengths $2$ and $3$, and has cycles of lengths $3$ and $4$.  Thus, the stated formulas hold.

Assume that the results hold for $t$.  By Lemma \ref{lem:recurrences}, $\calP_{t+1}=\{2\}\cup\{5+3i:0\le i\le t-1\}\cup\{3t+3\}=\{2+3i:0\le i\le t\}\cup\{3(t+1)\}$.  Similarly,
\begin{align*}
\calD_{t+1}&=\{3,4\}\cup\{5+3i,6+3i:0\le i\le t-2\}\cup\{5+3i:0\le i\le t-1\}\cup\{3t+3\}\\
&=\{3,4\}\cup\{5+3i,6+3i:0\le i\le t-1\}.
\end{align*}
  This completes the induction.
\end{proof}

\begin{theorem}\label{thm:ghm-counterexample}
For every $t\ge 1$, the graph $G_t$ is $4$-critical, has $3t+1$ vertices, and 
\begin{equation}\label{eq:intro-cycle-set}
\calL(G_t)
 =
\{3,4\}\cup\{5+3i,6+3i:0\le i\le t-2\}\cup\{3t+1\}.
\end{equation}  Moreover $\rho(G_1)=2$, $\rho(G_2)=5$, and $\rho(G_t)=4$ for all $t\ge 3$.
\end{theorem}

\begin{proof}
Criticality and the vertex count follow from Proposition \ref{prop:hajos-critical-size}.  By Lemmas \ref{lem:recurrences} and \ref{lem:explicit-pd},
\begin{align*}
\calL(G_t)&=\calD_t\cup(\calP_t+1)\\
&=\{3,4\}\cup\{5+3i,6+3i:0\le i\le t-2\}\cup\{3+3i:0\le i\le t-1\}\cup\{3t+1\}.
\end{align*}
The set $\{3+3i:0\le i\le t-1\}$ contributes $3$, already present, and contributes $6,9,\ldots,3t$, which are contained in $\{6+3i:0\le i\le t-2\}$ when $t\ge 2$.  The case $t=1$ is immediate.  This proves \eqref{eq:intro-cycle-set}.

It remains to determine $\rho(G_t)$.  For $t=1,2$, we have
$
\mathcal L(G_1)=\{3,4\}$ and $ \mathcal L(G_2)=\{3,4,5,6,7\}.
$
Now assume $t\ge 3$.  By the formula for $\mathcal L(G_t)$, the lengths
$7+3i~(0\le i\le t-3)$
are missing.  Hence, no interval of consecutive lengths can contain one of these gaps.  The longest initial interval of consecutive lengths is $\{3,4,5,6\}$; any later interval before the end has length at most $2$, and the final block has length at most $3$, which is $\{3t-1,3t,3t+1\}$.  Consequently,
$\rho(G_t)=4$
for all $t\ge 3$.
\end{proof}

\begin{proof}[\bf Proof of Theorem~\ref{thm:ghm-false}]
If such a function $f_3(n)$ existed, then Theorem~\ref{thm:ghm-counterexample} gives $f_3(3t+1)\le4$ for all $t\ge3$, contradicting $f_3(n)\to\infty$ along the subsequence $n=3t+1$.
\end{proof}

\section{Voss's conjecture and bounded-chord $k$-critical graphs}\label{sec:voss}

In this section, we disprove Conjecture~\ref{prob:voss} for every $k\ge 5$.  
We now describe the counterexample $H_{k,m}$. 
The graph consists of a clique $Q$, a chain of spine pentagons (the definition will be given later), and attached pentagons used to impose coloring restrictions.

\begin{cons}
Fix integers $k\ge 4$, $m\geq 1$ and set $p=k-3$.  Let
$Q=\{q_1,\ldots,q_p\}$
be a clique.  For each $i\in\{0,\ldots,m\}$, take a pentagon
$S_i=a_i b_i c_i d_i t_i a_i .$
We call the pentagons $S_i$ the spine, and the edges $c_i a_{i+1}$, $0\le i<m$, the spine edges.  The vertices of the spine pentagons are called \emph{spine vertices}.

For a vertex $x$ of some $S_i$ and $j\in\{1,\ldots,p\}$, let $L_j(x)$ be the pentagon $ruvwzr$ together with the edge $xr$. Add all edges from $u,v,w,z$ to $Q$, and all edges from $r$ to $Q\setminus\{q_j\}$. These are all edges involving this attached pentagon, apart from its pentagon edges and $xr$.

Leaves are attached according to the following pattern.  For every $0\le i\le m$ and every $j\in\{1,\ldots,p\}$, attach
$L_j(b_i),L_j(d_i),L_j(t_i).$
Also attach $L_j(a_0)$ and $L_j(c_m)$ for all $j$.  Finally, attach $L_j(a_i)$ for $1\le i\le m$ and $2\le j\le p$, and attach $L_j(c_i)$ for $0\le i<m$ and $2\le j\le p$.

The resulting graph is denoted by $H_{k,m}$.
The number of leaves  is $m(5p-2)+5p$, and therefore
\begin{equation}\label{eq:H-order}
\abs{V(H_{k,m})}=p+5(m+1)+5\bigl(m(5p-2)+5p\bigr)=(25k-80)m+26k-73.
\end{equation}
\end{cons}
\begin{lemma}\label{lem:leaf-forcing}
Let $G$ be a graph containing $Q$ and a leaf $L_j(x)$, including all edges from $L_j(x)$ to $Q$. Let $r$ be the vertex of $L_j(x)$ adjacent to $x$. In any proper $(k-1)$-coloring of $G$, the vertex $r$ receives the same color as $q_j$. Consequently, $x$ cannot receive the color of $q_j$.
\end{lemma}

\begin{proof}
Since $\lvert Q\rvert=p=k-3$ and there are $k-1=p+2$ colors, exactly two colors are absent from $Q$; call them $\alpha$ and $\beta$. Since $u,v,w,z$ are adjacent to all vertices of $Q$, they are restricted to $\alpha$ and $\beta$.  Since the colors alternate along the path $u v w z$, $u$ and $z$ receive different colors.  
Because $r$ is adjacent to both $u$ and $z$, it cannot use $\alpha$ or $\beta$.  
Moreover, $r$ is adjacent to every vertex of $Q\setminus\{q_j\}$; therefore $r$ receives the color of $q_j$.  Since $xr$ is an edge, $x$ cannot receive the color of $q_j$.
\end{proof} 

By Lemma~\ref{lem:leaf-forcing}, a leaf $L_j(x)$ forces $x$ to avoid the color of $q_j$.
This gives the following corollary.

\begin{corollary}\label{cor:leaf-restrictions}
Suppose that $H_{k,m}$ has a proper $(k-1)$-coloring.  Since $Q=\{q_1,\ldots,q_p\}$ is a clique, its vertices receive $p$ distinct colors.  Let $\alpha$ and $\beta$ be the two colors not used on $Q$.  Then the following hold.
\begin{enumerate}[label=\textup{(\roman*)},nosep]
\item For every $0\le i\le m$, the vertices $b_i,d_i,t_i$ use only $\alpha$ and $\beta$.  The same holds for $a_0$ and $c_m$.
\item For every $1\le i\le m$, the vertex $a_i$ avoids the colors of $q_2,\ldots,q_p$; thus $a_i$ may receive only $\alpha$, $\beta$, or the color of $q_1$.
\item For every $0\le i\le m-1$, the vertex $c_i$ avoids the colors of $q_2,\ldots,q_p$; thus $c_i$ may receive only $\alpha$, $\beta$, or the color of $q_1$.
\end{enumerate}
\end{corollary}

\begin{lemma}\label{lem:not-kminusone}
The graph $H_{k,m}$ is not $(k-1)$-colorable.
\end{lemma}

\begin{proof}
Suppose that $H_{k,m}$ has a proper $(k-1)$-coloring, and let
$\alpha,\beta$ be as in Corollary~\ref{cor:leaf-restrictions}.

We first show that $c_i$ receives the color of $q_1$ for every
$0\le i\le m-1$.  For $i=0$, Corollary~\ref{cor:leaf-restrictions}(i)
implies that $a_0,b_0,d_0,t_0$ use only $\alpha$ and $\beta$.
Since $S_0=a_0b_0c_0d_0t_0a_0$ is a pentagon, $c_0$ cannot use
$\alpha$ or $\beta$.  Thus, $c_0$ receives a color used on $Q$.
By Corollary~\ref{cor:leaf-restrictions}(iii), $c_0$ avoids the colors
of $q_2,\ldots,q_p$.  Hence, $c_0$ receives the color of $q_1$.

Now let $1\le i\le m-1$, and suppose that $c_{i-1}$ receives the color
of $q_1$.  Since $c_{i-1}a_i$ is an edge, $a_i$ avoids the color of
$q_1$.  By Corollary~\ref{cor:leaf-restrictions}(ii), $a_i$ also avoids
the colors of $q_2,\ldots,q_p$.  Hence $a_i$ uses either $\alpha$ or
$\beta$.  Corollary~\ref{cor:leaf-restrictions}(i) gives the same
restriction for $b_i,d_i,t_i$.  Therefore, the pentagon $S_i$ forces
$c_i$ to receive a color used on $Q$.  Since $c_i$ avoids the colors
of $q_2,\ldots,q_p$ by Corollary~\ref{cor:leaf-restrictions}(iii), it
receives the color of $q_1$.  This proves the claim.

Finally, $c_{m-1}$ receives the color of $q_1$.  Since $c_{m-1}a_m$
is an edge, $a_m$ avoids the color of $q_1$.  By
Corollary~\ref{cor:leaf-restrictions}(ii), $a_m$ also avoids the colors
of $q_2,\ldots,q_p$, and hence uses either $\alpha$ or $\beta$.
Together with Corollary~\ref{cor:leaf-restrictions}(i), this shows that
$a_m,b_m,d_m,t_m$ all use only $\alpha$ and $\beta$.  The pentagon
$S_m$ then forces $c_m$ to receive a color used on $Q$, contradicting
Corollary~\ref{cor:leaf-restrictions}(i), which says that $c_m$ uses only
$\alpha$ and $\beta$.
\end{proof}

Now we prove that $H_{k,m}$ is $k$-critical.
We first record the coloring property of the construction.  After the clique is removed, the remaining graph has a block structure.
A graph is called a \emph{cactus} if any two cycles have at most one vertex
in common. Equivalently, every block of the graph is either an edge or a
cycle. We shall use the following standard fact.

\begin{fact}\label{fact:cactus-3-colorable}
Every cactus is $3$-colorable.
\end{fact}

After deleting $Q$, the spine pentagons form a chain joined by bridges, and each leaf pentagon is attached by a single bridge; hence any two cycles meet in at most one vertex.

\begin{lemma}\label{lem:k-colorable}
The graph $H_{k,m}$ is $k$-colorable.
\end{lemma}

\begin{proof}
Color the clique $Q$ with $p$ distinct colors. Since $H_{k,m}-Q$ is a cactus whose blocks are edges and pentagons, by Fact \ref{fact:cactus-3-colorable}, $H_{k,m}-Q$ is $3$-colorable. Using three colors disjoint from the colors used in $Q$ gives a $k$-coloring of $H_{k,m}$.
\end{proof}

\begin{lemma}\label{lem:critical}
Every proper subgraph of $H_{k,m}$ is $(k-1)$-colorable.  Hence $H_{k,m}$ is $k$-critical.
\end{lemma}

\begin{proof}
We first compute the degrees of vertices outside $Q$. Each vertex of a spine pentagon has two neighbors on that pentagon.
For $0\leq i\leq m$,  the vertices $b_i,d_i,t_i$, together with $a_0$ and $c_m$, are attached to all $p$ leaves. Hence, their degree is $2+p$.
 For $1\le i\le m$, the vertex $a_i$ has one bridge to the spine and is attached to $p-1$ leaves; similarly, for $0\le i<m$, the vertex $c_i$ has one spine bridge and $p-1$ leaves.  Thus, these vertices also have degree $2+1+(p-1)=p+2$.
 Finally, inside each leaf the vertices $u,v,w,z$ have two neighbors on the leaf pentagon and are adjacent to all $p$ vertices of $Q$, while $r$ has two neighbors on the leaf pentagon, the attachment vertex $x$, and $p-1$ neighbors in $Q$.  Hence every vertex outside $Q$ has degree $p+2=k-1$.

We now show that deleting any edge leaves a $(k-1)$-colorable graph.  
First suppose $e\notin E(Q)$ and pick an endpoint $x$ of $e$ with $x\notin Q$.  Since $d_{H_{k,m}}(x)=k-1$, we have
$d_{H_{k,m}-e}(x)=k-2.$
Let $F=H_{k,m}-Q$.  The spine pentagons are connected by the edges $c_i a_{i+1}$ for $0\le i<m$, and every leaf pentagon is attached to a spine vertex; hence $F$ is connected.
Let $T$ be a spanning tree of $F$ rooted at $x$.  Delete the vertices of $F$ in non‑decreasing order of their distance from $x$ in $T$.  The root $x$ has degree $k-2$ in $H_{k,m}-e$.  For any vertex $v\neq x$ of $F$ with parent $u$ in $T$: if $v$ is an endpoint of $e$, then
$d_{H_{k,m}-e}(v)=k-2.$
Otherwise $uv\neq e$, so $uv$ belongs to $H_{k,m}-e$.  Because $u$ is deleted before $v$, the degree of $v$ when it is deleted is at most $d_{H_{k,m}}(v)-1=k-2.$
After all vertices of $F$ have been removed, only the clique $Q\cong K_p$ remains; each of its vertices has degree $p-1=k-4\le k-2$.  Thus every vertex has at most $k-2$ neighbours still present at the moment of deletion.  Colouring the vertices in the reverse order therefore yields a proper $(k-1)$-colouring of $H_{k,m}-e$.

Now suppose $e\in E(Q)$.  Then $Q-e$ is $(p-1)$-colorable, because the endpoints of $e$ may share a color.  The graph $H_{k,m}-Q$ is a cactus, hence $3$-colorable by Fact~\ref{fact:cactus-3-colorable}.  Using disjoint color sets for $Q-e$ and $H_{k,m}-Q$ gives a proper coloring of $H_{k,m}-e$ with
$
(p-1)+3 = p+2 = k-1
$
colors.  Hence $H_{k,m}-e$ is also $(k-1)$-colorable.

We have shown that every edge-deleted subgraph is $(k-1)$-colorable. Now let $J$ be an arbitrary proper subgraph of $H_{k,m}$.  If $J$ has the same vertex set as $H_{k,m}$, then it misses some edge $e$, so $J\subseteq H_{k,m}-e$.  If $J$ omits a vertex $v$, pick an edge $e$ incident to $v$ (such an edge exists because there are no isolated vertices); then $J\subseteq H_{k,m}-e$.  In either case $J$ is a subgraph of a $(k-1)$-colorable graph, hence itself $(k-1)$-colorable.

By Lemmas~\ref{lem:not-kminusone} and~\ref{lem:k-colorable}, $\chi(H_{k,m})=k$.  Since every proper subgraph is $(k-1)$-colorable, $H_{k,m}$ is $k$-critical.
\end{proof}

\begin{lemma}\label{lem:chord-bound}
Every cycle $C$ of $H_{k,m}$ has at most
$\binom{p}{2}+10p^{2}+60p$
chords.
\end{lemma}

\begin{proof}
Set $T=H_{k,m}-Q$. Its blocks are the pentagons $S_i$, the attached pentagons $L_j(x)$, and the bridge or attachment edges between them.

First suppose $C\cap Q=\varnothing$. Then $C$ is a cycle in $T$. Since the only cycles of $T$ are the spine pentagons and the leaf pentagons, $C$ is one of these pentagons. Moreover, none of these pentagons has a chord in $H_{k,m}$. Hence $C$ has no chord.

Now suppose $C\cap Q \neq \varnothing$.  Delete the vertices of $Q$ from $C$ and then obtain a union of nonempty paths in $T$:
$P_1 \cup \dots \cup P_t .$
Since each $P_i$ lies between two consecutive vertices of $C\cap Q$ along $C$, we have $t \le |C\cap Q| \le p$.

We first bound the number of chords incident with $Q$.  The number of edges with both ends in $Q$ is at most $\binom{p}{2}$.  Next, we count edges between $Q$ and $C-Q$.  In $T$, only vertices of leaf pentagons are adjacent to $Q$.  Each leaf pentagon is attached to the rest of $T$ by a single cut vertex.  Hence, if a path $P_i$ meets a leaf pentagon $L$, then $L$ contains an endpoint of $P_i$; otherwise $P_i$ would enter and leave $L$ through the same cut vertex.  Thus, the paths $P_1,\dots,P_t$ together meet at most $2t$ leaf pentagons.  Since each leaf pentagon has five vertices, at most $10t$ vertices of $C-Q$ are adjacent to $Q$.  Therefore, the number of edges between $Q$ and $C-Q$ is at most $10pt$, and the number of chords incident with $Q$ is at most
$\binom{p}{2}+10pt.$

We now bound the chords with both endpoints in $T$. Such a chord is either an edge of a spine or leaf pentagon, a leaf attachment edge, or a spine edge $c_i a_{i+1}$.  The cycle $C$ meets at most $2t$ leaf pentagons, and these leaf pentagons are attached at vertices lying on at most $2t$ spine pentagons.  Counting all pentagon edges in these leaf and spine pentagons gives at most
$5(2t+2t)=20t$
possible chords of this type.  No other spine pentagon contributes a pentagon-edge chord: if a path $P_j$ meets such a spine pentagon $S_i=a_i b_i c_i d_i t_i a_i$, then it enters through one of $a_i,c_i$ and leaves through the other, so it follows one of the two $a_i c_i$-paths on $S_i$; every unused edge of $S_i$ then has an endpoint outside $C$.

A leaf attachment edge can be a chord only when $C$ meets the corresponding leaf pentagon.  Finally, consider a spine edge
$e_i=c_i a_{i+1},$ $0\le i<m.$
For each $j\in\{1,\ldots,t\}$, let
\[
I_j=\bigl\{h\in\{0,\ldots,m\}:
V(P_j)\cap V(S_h)\neq\varnothing\bigr\}.
\]
Since the spine pentagons form a linear chain and every spine edge is
a bridge of $T$, each nonempty set $I_j$ is an interval of
consecutive integers. If
$I_j=\{\ell_j,\ell_j+1,\ldots,r_j\},$
we call
$c_{\ell_j-1}a_{\ell_j}$ when $\ell_j>0,$
and
$c_{r_j}a_{r_j+1}$
when $r_j<m$
the boundary spine edges of $I_j$.

Suppose that $e_i$ is a chord of $C$. Its endpoints cannot lie on
the same path $P_j$, since the subpath of $P_j$ joining $c_i$ to
$a_{i+1}$ would then have to use the bridge $e_i$, contradicting
$e_i\notin E(C)$. Hence
$c_i\in V(P_r)$
and
$a_{i+1}\in V(P_s)$
for some distinct $r,s$. The path $P_r$ cannot meet $S_{i+1}$,
and the path $P_s$ cannot meet $S_i$, since otherwise the
corresponding path would have to use $e_i$. It follows from the
interval property that
$i=\max I_r$
and
$i+1=\min I_s.$
Thus, $e_i$ is a boundary spine edge of both $I_r$ and $I_s$.
Since each nonempty interval $I_j$ has at most two boundary spine
edges, the number of spine-edge chords is at most $2t$.

Consequently, the number of chords with both endpoints in $T$ is at
most
$20t+2t+2t=24t<60t.$
Combining this with the estimate for chords incident with $Q$, and
using $t\le p$, we obtain $H_{k,m}$ has at most 
 $$\binom{p}{2}+10pt+60t
 \le \binom{p}{2}+10p^2+60p$$
chords.
\end{proof}

\begin{proof}[\bf Proof of Theorem \ref{thm:voss-fixed-intro}]
The order is computed in \eqref{eq:H-order}.  Criticality follows from Lemma \ref{lem:critical}.  The chord bound follows from Lemma \ref{lem:chord-bound} with $p=k-3$.
\end{proof}

\begin{corollary}\label{cor:voss-false}
Conjecture \ref{prob:voss} is false.
\end{corollary}

\begin{proof}
Fix $k\ge 4$.  By Theorem~\ref{thm:voss-fixed-intro}, there exist $k$-critical graphs $H_{k,m}$ whose orders tend to infinity as $m$ grows, while every odd cycle in $H_{k,m}$ has at most
\[
\binom{k-3}{2}+10(k-3)^2+60(k-3)
\]
chords.  This upper bound depends only on $k$ and not on $m$.  Hence, for fixed $k$, the values $g_k(n)$ are bounded along an infinite sequence of orders $n=|V(H_{k,m})|$.  Therefore $g_k(n)$ does not tend to infinity, contradicting Conjecture~\ref{prob:voss}.
\end{proof}

\begin{remark}
For $k=4$, Theorem~\ref{thm:apssv} gives the sharper estimate
$g_4(20m+31)\le 10.$
Although Theorem~\ref{thm:apssv} yields a stronger bound for $k=4$, our construction shows that for every fixed $k\ge4$, an infinite family of $k$-critical graphs in which the number of chords on every odd cycle is bounded by a constant depending only on $k$.
\end{remark}

\section{On a conjecture of K\'ara--Kr\'al}

The purpose of this section is to prove $f(31,31)=8$: we first prove that every $31$-vertex graph with minimum degree at least $8$ contains a cycle with at least $31$ chords, and then give an example showing that minimum degree $7$ does not suffice.

Let $G$ be a connected graph. A vertex $v\in V(G)$ is called a \textit{cut vertex }of $G$ if $G-v$ has more connected components than $G$.
A \textit{block} of $G$ is a maximal connected subgraph of $G$ that has no cut vertex. 
A block consisting of a single edge is called a \emph{bridge block}.  Thus every block is either $2$-connected, a single vertex, or a bridge block.
A block $B^*$ of $G$ is called an \textit{end-block} if it contains exactly one cut vertex of $G$.

Let $B$ be a block of $G$. A vertex $v \in V(B)$ is called \textit{private to $B$} (or a \textit{private vertex} of $B$) if it is not a cut vertex of $G$.
The \textit{private vertex set} of a block $B$ is defined by
$Q(B)=\{v\in V(B): v\text{ is not a cut vertex of }G\}.$
Equivalently, $Q(B)$ is the set of vertices of $B$ which are private to $B$.
The \textit{private size} of a block $B$ is
$q(B)=|Q(B)|.$
We shall use the following fact. 


\begin{fact}\label{fact:private}
    If $v$ is private to a block $B$, then every neighbor of $v$ in $G$ belongs to $B$; that is,
$N_G(v)\subseteq V(B).$
Consequently,
$d_G(v)=d_B(v).$
In particular, if $\delta(G)\ge d$ and $v$ is private to $B$, then
$d_B(v)\ge d.$
\end{fact}

\subsection{Main tools}
In this subsection we collect the tools used to force a cycle with at least $31$ chords.

In \cite{Ash1985}, the proof of Theorem~\ref{thm:Ash1985} splits into two cases depending on whether $G$ is Hamiltonian.  If $G$ is Hamiltonian, the result follows from Lemma~\ref{lem:ash-counting}; otherwise, it follows from Lemma~\ref{lem:ash-terminal}.

\begin{lemma}[Lemma 2 in \cite{Ash1985}]
\label{lem:ash-counting}
Let $d\ge 3$.  Let $G$ be a Hamiltonian graph and $C$ be a Hamilton cycle of $G$. Set
$Y\subseteq V(G)$.  Suppose that
$d_G(y)\ge d $ for every $y\in Y.$
If either $|Y|\ge 2d$, or $Y$ contains $d$ independent vertices,
then $C$ has at least $d(d-2)$ chords.
\end{lemma}

Before stating Lemma \ref{lem:ash-terminal}, we give a terminal-cycle setup \cite{Ash1985}.

\vspace{2mm}
\noindent{\bf Terminal-cycle setup.}
Let $G$ be a $2$-connected non-Hamiltonian graph and let $v_0\in V(G)$.
Choose a longest path $P=v_0v_1\cdots v_\ell$ starting at $v_0$; among all such paths, select one that maximizes the length of the terminal cycle defined below. 

Since $P$ is maximal, every neighbor of $v_\ell$ is in $P$.  
Because $G$ is $2$-connected, $v_\ell$ has a neighbor on $P$ distinct from $v_{\ell-1}$.  
Hence the index $s=\min\{\,i : v_i v_\ell\in E(G)\,\}$ satisfies $s\le \ell-2$.  
Write
$u_m=v_s,\; u_{m-1}=v_{s+1},\; \dots,\; u_1=v_\ell .$
Then
$C = u_m u_{m-1}\cdots u_1 u_m$
is a cycle, called the \emph{terminal cycle} of $P$.  Set
$H = G[\{u_1,u_2,\ldots,u_m\}].$

A vertex $w\in V(H)\setminus\{u_m\}$ is \emph{accessible} if $H$ contains a Hamilton path from $u_m$ to $w$.  
Let $W(H)$ denote the set of accessible vertices.  
This construction is the \emph{terminal-cycle setup}.
We record it separately because it will be used in a rooted form below.

The following lemma is not a separate statement in \cite{Ash1985}: it is distilled from the proof of Ash's Theorem~1 (see the part of the case $G$ is non-Hamiltonian). 

\begin{lemma}[Terminal-cycle lemma \cite{Ash1985}]
\label{lem:ash-terminal}
In the terminal-cycle setup, let $d\ge 3$.  If
$u_1,u_{m-1}\in W(H)$
and
$d_H(a)=d_G(a)\ge d$  for all $ a\in W(H)$,
then $G$ contains a cycle with at least $d(d-2)$ chords.
\end{lemma}

To apply the terminal-cycle lemma when one vertex is allowed to be exceptional, we need the following lemma.

\begin{lemma}\label{lem:rooted-ash-verification} Let $G$ be a $2$-connected non-Hamiltonian graph, and let $z\in V(G)$. Let $P,C,H,W(H)$ be defined by the terminal-cycle setup above, with initial vertex $v_0=z$. Then 
(i) $u_1,u_{m-1}\in W(H)$; (ii) $W(H)\cap\{z\}=\emptyset$; (iii) $d_H(w)=d_G(w)$ for all $w\in W(H)$.
\end{lemma}
\begin{proof}
By the definition of $W(H)$, assertions (i) and (ii) hold.  It remains to prove (iii).
Let $w\in W(H)$ and let $P_0$ be a Hamilton path of $H$ from $u_m$ to $w$.  We show that $w$ has no neighbor in $V(G)\setminus V(H)$.

Suppose, to the contrary, that $w$ has a neighbor $w'\notin V(H)$.  
If $w'\notin\{v_0,v_1,\dots,v_{s-1}\}$, then $v_0v_1\cdots v_{s-1}u_mP_0ww'$ is a path starting at $z=v_0$ that is strictly longer than $P$, contradicting the maximality of $P$.  Hence every neighbor of $w$ outside $H$ is of the form $v_i$ for some $i<s$.
Choose such a neighbor $v_i$ and let $wP_0^{-1}u_m$ denote the path $P_0$ traversed backwards from $w$ to $u_m$.

If $i<s-1$, then $v_0v_1\cdots v_i\, wP_0^{-1}u_m\, v_{s-1}v_{s-2}\cdots v_{i+1}$ is a path starting at $z=v_0$.  Its terminal segment together with the edge $v_{i+1}v_i$ yields the cycle $v_i\, wP_0^{-1}u_m\, v_{s-1}v_{s-2}\cdots v_{i+1}v_i$.  This cycle is strictly longer than $C$, a contradiction.

If $i=s-1$, then $v_0v_1\cdots v_{s-1}\, wP_0^{-1}u_m$ is a path starting at $z=v_0$.  Its terminal segment together with the edge $u_mv_{s-1}$ yields the cycle $v_{s-1}\, wP_0^{-1}u_m v_{s-1}$, which contains all vertices of $H$ and also $v_{s-1}$.  This cycle is strictly longer than $C$, a contradiction.

Thus, $w$ has no neighbor in $V(G)\setminus V(H)$.  Consequently, $d_H(w)=d_G(w)$ for every $w\in W(H)$.  This proves (iii).
\end{proof}

\begin{proof}[Proof of Theorem \ref{thm:one-exception-ash}.]
Suppose that $G$ is not Hamiltonian.
Let $P,C,H,W(H)$ be obtained from the terminal-cycle setup with $v_0=z$.
By Lemma~\ref{lem:rooted-ash-verification}, we have $u_1,u_{m-1}\in W(H)$, $W(H)\cap\{z\}=\emptyset$ and $d_H(a)=d_G(a)$ for all $a\in W(H)$.
Hence, $d_H(w)=d_G(w)\ge d$ for all $w\in W(H)$.
By Lemma~\ref{lem:ash-terminal}, $G$ contains a cycle with at least $d(d-2)$ chords.

It remains to consider the Hamiltonian case. Let $C$ be a Hamilton cycle of $G$.
Put $Y=V(G)\setminus\{z\}$. Then $|Y|\ge 2d$, and
$d_G(y)\ge d$ for every $y\in Y$. Hence, Lemma~\ref{lem:ash-counting}
implies that $C$ has at least $d(d-2)$ chords.
\end{proof}

We recall the Bondy--Chv\'{a}tal closure theorem. Let $G$ be an $n$-vertex graph. If $u$ and $v$ are nonadjacent vertices with $d_G(u)+d_G(v)\ge n$, adding the edge $uv$ is called a \textit{closure operation}. The \emph{closure} $\operatorname{cl}(G)$ (or $\operatorname{cl}_n(G)$) is obtained by repeatedly performing such operations until no eligible pair remains; the result is independent of the order of operations.

\begin{theorem}[Bondy--Chv\'{a}tal \cite{BondyChvatal1976}]\label{thm:closure}
Let $G$ be an $n$-vertex graph, and let $u,v\in V(G)$ be nonadjacent vertices with $d_G(u)+d_G(v)\ge n$. (i) Then $G$ is Hamiltonian if and only if $G+uv$ is Hamiltonian. (ii) $G$ is Hamiltonian if and only if its Bondy--Chvátal closure $\operatorname{cl}(G)$ is Hamiltonian. (iii) If $\operatorname{cl}(G)$ is complete, then $G$ is Hamiltonian.
\end{theorem}

\subsection{Main results}
We next analyze a minimal counterexample.  The first lemma treats a $2$-connected subgraph with two exceptional vertices; later this subgraph will appear as the unique internal block between two end-blocks.

\begin{lemma}
\label{lem:two-exceptional}
Let $G$ be 2-connected with $V(G)=Q\cup\{x,y\}$, where $9\le |Q|=q\le13$.  Suppose $d_G(v)\geq 8$ for $v\in Q$, and put $s=d_G(x)+d_G(y)$.  If either $q\ge11$, or $q=10$ and $s\ge6$, or $q=9$ and $s\ge12$, then $G$ contains a cycle with at least $31$ chords.
\end{lemma}

\begin{proof}
Since for any two nonadjacent vertices $u,v\in Q$, we have $d_G(u)+d_G(v)\ge 16\ge q+2=|V(G)|$, we may repeatedly apply the closure operation to nonadjacent pairs in $Q$ to obtain the graph $H$.  
We distinguish two cases depending on whether $H$ is Hamiltonian.

\medskip
\noindent
\textbf{Case 1: $H$ is Hamiltonian.} 
 \medskip

By Theorem~\ref{thm:closure}, $G$ is also Hamiltonian. Let $C$ be a Hamilton cycle of $G$.
Since every edge of $G$ not in $C$ is a chord of $C$,
$\operatorname{ch}_G(C)=|E(G)|-|V(G)| \ge {(8q+s)}/{2}-(q+2)=3q-2+{s}/{2}.$
If $q=9$ and $s\ge12$, then $3q-2+{s}/{2}\ge 31$.
If $q=10$ and $s\ge6$, then $3q-2+{s}/{2}\ge 31$.
If $q\ge 11$, then $s\ge 4$ as $G$ is $2$-connected, which implies that $3q-2+{s}/{2}\ge 33$.

\medskip
\noindent
\textbf{Case 2: $H$ is not Hamiltonian.} 
 \medskip

We first prove that $q\geq 11$ by using the following claim.
\begin{claim}
    $xy\notin E(H)$ and $N_H(x)\cap Q=N_H(y)\cap Q=\{a,b\}$ for distinct $a,b\in Q$.
\end{claim}

\begin{proof}
By construction, all nonadjacent pairs in $Q$ are added, so $H[Q]\cong K_q$. By Theorem~\ref{thm:closure}, $G$ is Hamiltonian if and only if $H$ is Hamiltonian.
    Set $X=N_H(x)\cap Q$, $Y=N_H(y)\cap Q$.
Since $H$ is 2-connected, $X,Y\neq\varnothing$. Moreover, there are distinct $a\in X$, $b\in Y$; otherwise $X=Y=\{a\}$ and removing $a$ separates $\{x,y\}$, a contradiction.

If $xy\in E(H)$, a Hamilton path in $H[Q]$ from $b$ to $a$ gives the Hamilton cycle, a contradiction.  Hence $xy\notin E(H)$.

Since $H$ is $2$-connected and $xy\notin E(H)$, both $|X|$ and $|Y|$ are at least 2.
If $X\cap Y=\varnothing$, take distinct $a,b\in X$ and $c,d\in Y$.  Since $H[Q\setminus\{b,c\}]$ is complete, there exists a Hamilton path $P_1$ from $d$ to $a$, yielding the cycle $a x b c y dP_1a$, a contradiction. Thus $X\cap Y\neq\varnothing$.
Pick $b\in X\cap Y$.  If there exist $a\in X\setminus\{b\}$ and $c\in Y\setminus\{b\}$ with $a\neq c$, then $H[Q\setminus\{b\}]$ has a Hamilton path $P_2$ from  $c$ to $a$, giving the cycle $a x b y c P_2a$, a contradiction.  Hence $X\setminus\{b\}=Y\setminus\{b\}$ and this set has size at most $1$.  Because $|X|,|Y|\ge2$, we must have $X\setminus\{b\}=Y\setminus\{b\}=\{a\}$ for some $a\in Q\setminus\{b\}$.  Thus $X=Y=\{a,b\}$.
\end{proof}

Since the closure operation only adds edges in $Q$, we have $N_G(x)=N_G(y)=\{a,b\}$. Then $d_G(x)=d_G(y)=2$ and $s=4$.  Thus, $q\ge11$.

We next prove that $G-y$ is Hamiltonian. 
Since for any two nonadjacent vertices $u,v\in Q$, we have $d_{G-y}(u)+d_{G-y}(v)\ge14\ge q+1=\lvert V(G-y)\rvert$, we may repeatedly apply the closure operation to nonadjacent pairs in $Q$.
In the resulting graph, $Q$ is a clique and $x$ is adjacent to $a,b\in Q$.  A Hamilton path in $Q$ from $a$ to $b$ together with $xa$ and $xb$ gives a Hamilton cycle; thus $G-y$ is Hamiltonian.

Let $C$ be a Hamilton cycle of $G-y$. 
Since $s=4$ and $d_G(y)=2$, $|E(G-y)|=|E(G)|-d_G(y)\geq (8q+4)/2-2\ge 4q$. Hence,
$\operatorname{ch}_{G-y}(C)=|E(G-y)|-|V(G-y)|\ge 4q-(q+1)=3q-1\ge 32,$
as $q\ge11$. Thus, $C$ has at least $32$ chords in $G$, yielding a cycle with at least $31$ chords.
\end{proof}

The next lemma shows that a small end-block in a counterexample is very restricted.

\begin{lemma}
\label{lem:small-endblock}
Let $G$ be a $31$-vertex graph with $\delta(G)\ge8$ and no cycle with at least $31$ chords.  Let $B^*$ be an end-block and $x$ be the cut vertex of $B^*$.  If $q(B^*)\le15$, then $q(B^*)\in\{8,9,10\}$.  Moreover, if $q(B^*)=8$, then $B^*\cong K_9$ and $d_{B^*}(x)=8$; if $q(B^*)=9$, then $2\le d_{B^*}(x)\le8$; if $q(B^*)=10$, then $d_{B^*}(x)=2$, $|E(B^*)|=41$, and every vertex of $Q(B^*)$ has degree exactly $8$ in $B^*$.
\end{lemma}

\begin{proof}
Let $x$ be the cut vertex of $B^*$.
By Fact~\ref{fact:private}, every $v\in Q(B^*)$ has all its neighbors in $B^*$, so $d_{B^*}(v)=d_G(v)\ge 8$. Moreover, $q(B^*)\ge 8$. Now assume $q(B^*)\leq 15$. Then we prove $q(B^*)\in\{8,9,10\}$.

Since every private vertex of $B^*$ has degree at least $8$, $B^*$ cannot be a bridge block. Hence $B^*$ is $2$-connected. Then $x$ has two distinct neighbors $a,b\in Q(B^*)$.
Since for any two nonadjacent vertices $u,v\in Q(B^*)$, we have $d_{B^*}(u)+d_{B^*}(v)\ge 16\ge q(B^*)+1=|V(B^*)|$, we may repeatedly apply the closure operation to nonadjacent pairs in $Q(B^*)$ to obtain the graph $H$.
By Theorem \ref{thm:closure},
$H[Q(B^*)]$ is a clique.
Thus, there exists 
a Hamilton path $P$ in $H[Q(B^*)]$ from $a$ to $b$. 
Then $xaPbx$ is a Hamilton cycle of $H$. Hence, by Theorem~\ref{thm:closure}, $B^*$ is Hamiltonian.

Let $C$ be a Hamilton cycle of $B^*$.  Since $G$ has no cycle with $31$ chords,
  $|E(B^*)|-|V(B^*)|\le30.$
On the other hand,
  $|E(B^*)|\ge \frac{8q(B^*)+d_{B^*}(x)}{2}.$
Thus,
\[
6q(B^*)+d_{B^*}(x)\le 62.
\]
Since $d_{B^*}(x)\ge2$, we have $q(B^*)\le10$.  Together with $q(B^*)\ge8$, this gives
$q(B^*)\in\{8,9,10\}.$

If $q(B^*)=8$, then $d_{B^*}(v)\ge8=|V(B^*)|-1$ for every $v\in Q(B^*)$. Hence $B^*\cong K_9$ and $d_{B^*}(x)=8$.

If $q(B^*)=9$, then $54+d_{B^*}(x)\le62$, so $d_{B^*}(x)\le8$. Since $B^*$ is $2$-connected, $2\le d_{B^*}(x)\le8$.

If $q(B^*)=10$, then $60+d_{B^*}(x)\le62$.  As $d_{B^*}(x)\ge2$, we get $d_{B^*}(x)=2$. Moreover,
$|E(B^*)|\ge \frac{8\cdot10+2}{2}=41
$ and $
|E(B^*)|\le |V(B^*)|+30=41.$
Thus $|E(B^*)|=41$. Hence
$\sum_{v\in Q(B^*)} d_{B^*}(v) = 2|E(B^*)|-d_{B^*}(x) = 80.$
Since every vertex of $Q(B^*)$ has degree at least $8$ in $B^*$, it follows that each vertex of $Q(B^*)$ has degree exactly $8$.
\end{proof}

The next proposition handles the complementary case: a large block with only one exceptional vertex already forces many chords.

\begin{proposition}\label{prop:large-terminal}
Let $G$ be 2-connected and let $x\in V(G)$.  If $|V(G)|\ge15$ and $d_G(v)\ge8$ for every $v\in V(G)\setminus \{x\}$, then $G$ contains a cycle with at least $31$ chords.
\end{proposition}

\begin{proof}
Let $H=G-x$.  Since $G$ is 2-connected, $H$ is connected.  Moreover, for
all $v\in V(H)$,
$d_H(v)\ge d_G(v)-1\ge 7.$
If $H$ is 2-connected, there exists a cycle with at
least $7(7-2)=35$ chords by Theorem \ref{thm:Ash1985}, and we are done.

Suppose $H$ is not $2$-connected.  Let $B^*$ be an end-block of $H$ with cut vertex $y$, and set $Q(B^*) = V(B^*) \setminus \{y\}$.
For every $v \in Q(B^*)$, we have $d_{B^*}(v) = d_H(v) \ge 7$.  Hence $B^*$ is not a single edge, and therefore $B^*$ is $2$-connected.  Moreover, $|Q(B^*)|\ge 7$.
Finally, $x$ has a neighbor in $Q(B^*)$; otherwise $y$ would separate $Q(B^*)$ from $x$ in $G$.

If $|Q(B^*)|\ge 14$, then $|V(B^*)|\ge 2\cdot7+1$.  Apply Theorem~\ref{thm:one-exception-ash} to $B^*$ with exceptional vertex $y$ and $d=7$. Then $B^*$ contains a cycle with at least $35$ chords.

If $|Q(B^*)|=13$, then for any nonadjacent $u,v\in Q(B^*)$,
$d_{B^*}(u)+d_{B^*}(v)\ge14=|V(B^*)|.$
Thus, we may repeatedly apply the closure operation to nonadjacent pairs in $Q(B^*)$, making it a clique.  Since $B^*$ is $2$-connected, $y$ has two neighbors in $Q(B^*)$.  Hence, the closure contains a Hamilton cycle through $y$ and $Q(B^*)$, and by Theorem~\ref{thm:closure}, $B^*$ is Hamiltonian.
Let $C$ be a Hamilton cycle of $B^*$.  Since $d_{B^*}(v)\ge7$ for all $v\in Q(B^*)$ and $d_{B^*}(y)\ge2$,
$2|E(B^*)|\ge 7|Q(B^*)|+2=93.$
Thus $|E(B^*)|\ge47$.  As $|V(B^*)|=14$, we obtain
$\operatorname{ch}_{B^*}(C)=|E(B^*)|-|V(B^*)|\ge47-14=33.$

Finally, suppose that $7\le |Q(B^*)|\le12$.  Choose two end-blocks $B^*_1$ and $B^*_2$.  For $i=1,2$, let $y_i$ be the cut vertex of $B^*_i$ in $H$, set
 $z_i\in Q(B^*_i)\cap N_G(x)$.

We first show that $B^*_i$ contains a Hamilton path from $y_i$ to $z_i$.
Add a new vertex $w_i$ adjacent only to $y_i$ and $z_i$.  In $B^*_i+w_i$, every nonadjacent pair $u,v\in Q(B^*_i)$ satisfies
$d_{B^*_i+w_i}(u)+d_{B^*_i+w_i}(v)\ge 14\ge q_i+2=|V(B^*_i+w_i)|.$
Thus, the closure makes $Q(B^*_i)$ a clique.  Since $B^*_i$ is $2$-connected, $y_i$ has a neighbor $a_i\in Q(B^*_i)\setminus\{z_i\}$.  Thus the closure contains a Hamilton cycle
$w_i z_i P a_i y_i w_i,$
with $P$ a Hamilton path in $Q(B^*_i)$ from $z_i$ to $a_i$.
  By Theorem~\ref{thm:closure}, $B^*_i+w_i$ is Hamiltonian. Since $w_i$ is adjacent only to $y_i$ and $z_i$, deleting $w_i$ from a
Hamilton cycle of $B^*_i+w_i$ gives a Hamilton path $P_i$ of $B^*_i$
with ends $y_i$ and $z_i$.

Since $B_1^*$ and $B_2^*$ are end-blocks of $H$,  deleting the private vertices of $B_1^*$ and $B_2^*$ from $H$ leaves a connected graph that still contains $y_1$ and $y_2$.  Thus there is a $y_1y_2$-path $P_0$ in $H-(Q(B_1^*)\cup Q(B_2^*))$.  Then
$C = xz_1P_1y_1P_0y_2P_2z_2x$
is a cycle of $G$.  For each $i$, the path $P_i$ uses exactly $q_i$ edges of $B^*_i$; all remaining edges of $B^*_i$ are chords of $C$.  Moreover,
$2|E(B^*_i)|\ge 7q_i+2.$
Consequently,
$
\operatorname{ch}_G(C)
\ge \sum_{i=1}^2 \bigl(|E(B^*_i)|-q_i\bigr) 
\ge \sum_{i=1}^2 \left(\frac{7q_i+2}{2}-q_i\right)
   = \frac{5(q_1+q_2)+4}{2}
   \ge 37,
$
as $q_1,q_2\ge7$. This gives a cycle with at least $31$ chords and completes the proof.
\end{proof}

\begin{proposition}\label{prop:global-small}
Let $G$ be a graph on $31$ vertices with $\delta(G)\ge 8$ and no cycle having at least $31$ chords.  Then $G$ contains an end-block $B^*$ such that
$q(B^*)\ge 16.$
\end{proposition}

\begin{proof}
Suppose not. We may assume $G$ is connected; if not, adding bridges between its components produces a graph $G'$ with $\delta(G')\ge\delta(G)\ge8$, and every cycle in $G'$ already lies in one component of $G$.  Hence a cycle with at least $31$ chords in $G'$ would also exist in $G$.  Moreover, by Theorem~\ref{thm:Ash1985}, $G$ is not $2$-connected.
Thus, it has at least two end-blocks.  By Lemma~\ref{lem:small-endblock}, each end-block has private size at least $8$. Since distinct end-blocks have disjoint private sets
and $|V(G)|=31$, $G$ has two or three end-blocks.
We distinguish two cases based on the number of end-blocks.

\medskip

\noindent
\textbf{Case 1: $G$ has three end-blocks $B^*_1,B^*_2,B^*_3$.} 
 \medskip

Let $Q(B^*_i)$ be the private set of $B^*_i$ with order $q_i$ and put $p=31-(q_1+q_2+q_3)$.  Since $q_i\ge8$, we have $p\le7$.

We first show that $p\leq 3$.  Suppose some internal block $B'$ has a nonempty private set $Q(B')$ and contains $r$ cut vertices ($r\ge2$). 
Then $|Q(B')|+r\leq p\leq 7$.
For every $v\in Q(B')$, Fact~\ref{fact:private} gives $d_{B'}(v)=d_G(v)\ge8$, while $d_{B'}(v)\le |Q(B')|-1+r$, so $|Q(B')|+r\ge9$, a contradiction. Hence no internal block has private vertices; every vertex  outside $\bigcup_iQ(B^*_i)$ is a cut vertex. We now prove a useful claim.
\begin{claim}\label{claim:cut-vertex}
    Each cut vertex $x$ is incident with at least one end-block.
\end{claim}

\begin{proof}
    Suppose not. Then all neighbors of $x$ can only be the other at most $p-1$ vertices.
    Since $p\leq 7$, $d_G(x)\leq p-1\leq 6$, a contradiction.
\end{proof}
Since each end-block has exactly one cut vertex, if $p\ge4$, some cut vertex is incident with no end-block, contradicting Claim \ref{claim:cut-vertex}.  Thus, $p\le3$.

Therefore, $q_1+q_2+q_3\ge28$.  Since each $q_i\in\{8,9,10\}$, at least one end-block has private size $10$.  Recall from Lemma~\ref{lem:small-endblock} that in an end-block of private size 10, the cut vertex has degree $2$ inside that block.
We now check the possible values of $(q_1,q_2,q_3)$.  By symmetry, assume $q_1 \ge q_2\ge q_3$.

If $p=1$, then $q_i=10$ for $i=1,2,3$.  The unique cut vertex is incident with all three end-blocks and has degree at most $2+2+2=6$, a contradiction.

If $p=2$, then $q_1=q_2=10$ and $q_3=9$.  Let $x,y$ be the two
cut vertices.  By Claim~\ref{claim:cut-vertex}, we may assume that $x$
is incident with two end-blocks.  Since there is no internal block with
private vertices, each of $x,y$ has at most one neighbor outside its
incident end-blocks.
If both end-blocks of private size $10$ are incident with $x$, then
Lemma~\ref{lem:small-endblock} gives
$ d_G(x)\le 2+2+1=5.$
Otherwise, $y$ is incident with an end-block of private size $10$, and
hence
       $ d_G(y)\le 2+1=3.$
In either case this contradicts $\delta(G)\ge8$.

If $p=3$, then either $q_1=q_2=10$, $q_3=8$, or $q_1=10$, $q_2=q_3=9$. By Claim \ref{claim:cut-vertex}, there exists a cut vertex $x$ incident with an end-block of private size 10. 
Then $d_G(x)\le 4$, a contradiction.

 \medskip
\noindent
\textbf{Case 2: $G$ has two end-blocks $B^*_1,B^*_2$.} 

 \medskip

 We next prove that there is a unique internal block with private vertices.

If no internal block has private vertices, then every internal block is a bridge block.  A cut vertex not contained in either end-block then belongs to exactly two bridge blocks and therefore has degree $2$, contradicting $\delta(G)\ge8$.  Hence every cut vertex lies in at least one end-block.  The two end-blocks therefore either share a cut vertex or are joined by a single bridge block.  In either case,
$|V(G)|\le 10+10+2=22,$
contradicting $|V(G)|=31$.  Thus at least one internal block has private vertices.
Since an internal block $B'$ with private set $Q(B')$ has exactly two cut vertices, and for $v\in Q(B')$, $d_{B'}(v)\le |Q(B')|-1+2=|Q(B')|+1$, while $d_{B'}(v)\ge8$, so $|Q(B')|\ge7$.
If there were two such internal blocks, then the two end-blocks and these
two internal blocks would form a chain of four blocks, requiring at least
three distinct cut vertices.
  The total number of vertices would then be at least $8+8+7+7+3 > 31$, a contradiction.  Therefore, there is a unique internal block $B$ with private vertices, and $B$ is $2$-connected.

Let $x_1,x_2$ be the two cut vertices of the unique internal block $B$ that contains private vertices, and set $Q(B)=V(B)\setminus\{x_1,x_2\}$.
By Fact~\ref{fact:private},
$d_B(v) = d_G(v) \ge 8 $ for every  $ v \in Q(B).$
Let $B^*_1,B^*_2$ be the two end-blocks.  For $i=1,2$, let $a_i$ be the cut vertex of $B^*_i$, let $q_i$ be the private size of $B^*_i$, and set
$\alpha_i = d_{B^*_i}(a_i).$
For each $B^*_i$, either $a_i = x_i$ or $a_i$ and $x_i$ are joined by exactly one bridge block. 
Define
\[
\varepsilon_i =
\begin{cases}
0, & \text{if } a_i = x_i,\\[2pt]
1, & \text{if } a_i \ne x_i.
\end{cases}
\]
Thus $\varepsilon_i$ indicates whether there is an extra cut vertex between $B^*_i$ and $B$.  Since $G$ has $31$ vertices, we have the $q_1+q_2$ private vertices of the two end-blocks, the $|Q(B)|$ private vertices of the internal block, the two cut vertices $x_1$ and $x_2$, and $\varepsilon_1+\varepsilon_2$ extra cut vertices.  Therefore
$31 = q_1 + q_2 + |Q(B)| + 2 + \varepsilon_1 + \varepsilon_2,$
or equivalently
\begin{equation}\label{eq:central-count}
    |Q(B)| = 29 - q_1 - q_2 - \varepsilon_1 - \varepsilon_2.
\end{equation}

\begin{claim}\label{claim:var}
    If $\varepsilon_i = 1$, then $d_B(x_i) \ge 7$. If $\varepsilon_i = 0$, then $d_B(x_i) \ge 2$. Moreover, if $q_i = 10$, then necessarily $\varepsilon_i = 0$, and in this case $d_B(x_i) \ge 6$.
\end{claim}

\begin{proof}
Suppose first that $\varepsilon_i = 1$.  
Then  $x_i$ also has exactly one neighbour outside $B$.  Since $d_G(x_i) \ge 8$, we obtain $d_B(x_i) \ge d_G(x_i) - 1 \ge 7$.
Now suppose that $\varepsilon_i = 0$.   Since $B$ is $2$-connected, $d_B(x_i) \ge 2$.  Moreover, the neighbours of $x_i$ in $B^*_i$ contribute exactly $\alpha_i$ to its degree outside $B$.  Hence $d_G(x_i) = d_B(x_i) + \alpha_i \ge 8$. Thus $d_B(x_i) \ge \max\{\,2,\; 8 - \alpha_i \,\}$.

Finally, if $\varepsilon_i=1$,
then  $a_i$ has exactly one neighbour outside $B^*_i$.  Hence $d_G(a_i) = \alpha_i + 1$.  Since $d_G(a_i) \ge 8$, it follows that $\alpha_i \ge 7$.  Therefore $q_i \neq 10$, because Lemma~\ref{lem:small-endblock} gives $\alpha_i = 2$ whenever $q_i = 10$.
Thus,  if $q_i = 10$, $\varepsilon_i = 0$.  Moreover, Lemma~\ref{lem:small-endblock} gives $\alpha_i = 2$, and therefore $d_B(x_i) \ge 8 - \alpha_i = 6$.
\end{proof}

We now verify the hypotheses of Lemma~\ref{lem:two-exceptional} for $B$, where the degree-sum parameter is $s=d_B(x_1)+d_B(x_2)$.  It suffices to check that $|Q(B)|\ge11$, or $|Q(B)|=10$ and $s\ge6$, or $|Q(B)|=9$ and $s\ge12$. 
We now check the possible values of $(q_1,q_2)$.  By symmetry, assume $q_1 \le q_2$. There are six choices for $(q_1,q_2)$.

If $(q_1,q_2) = (8,8)$, then by \eqref{eq:central-count}, $|Q(B)| = 13 - \varepsilon_1 - \varepsilon_2 \ge 11$.

If $(q_1,q_2) = (8,9)$, then $|Q(B)| = 12 - \varepsilon_1 - \varepsilon_2$.  Thus either $|Q(B)| \ge 11$, or $|Q(B)| = 10$.  If $|Q(B)|=10$, $\varepsilon_1 = \varepsilon_2 = 1$, so Claim \ref{claim:var} gives $s = d_B(x_1) + d_B(x_2) \ge 7 + 7 \ge 6$.

If $(q_1,q_2) = (8,10)$, Claim \ref{claim:var} implies $\varepsilon_2 = 0$.  Hence $|Q(B)| = 11 - \varepsilon_1$.  Thus either $|Q(B)| = 11$, or $|Q(B)| = 10$. If $|Q(B)|=10$, $\varepsilon_1 = 1$, and Claim \ref{claim:var} gives $s = d_B(x_1) + d_B(x_2) \ge 7 + 6 \ge 6$.

If $(q_1,q_2) = (9,9)$, then $|Q(B)| = 11 - \varepsilon_1 - \varepsilon_2$.  If $|Q(B)| = 11$, the first alternative of Lemma~\ref{lem:two-exceptional} holds.  If $|Q(B)| = 10$, exactly one of $\varepsilon_1,\varepsilon_2$ equals $1$, and Claim \ref{claim:var} gives $s \ge 7 + 2 \ge 6$.  If $|Q(B)| = 9$, then $\varepsilon_1 = \varepsilon_2 = 1$, and the claim gives $s \ge 7 + 7 \ge 12$.

If $(q_1,q_2) = (9,10)$, Claim \ref{claim:var} gives $\varepsilon_2 = 0$.  Hence $|Q(B)| = 10 - \varepsilon_1$.  If $|Q(B)| = 10$, then $\varepsilon_1 = 0$, and Claim \ref{claim:var} gives $s \ge 2 + 6 \ge 6$.  If $|Q(B)| = 9$, then $\varepsilon_1 = 1$, and Claim \ref{claim:var} gives $s \ge 7 + 6 \ge 12$.

Finally, if $(q_1,q_2) = (10,10)$, Claim \ref{claim:var} gives $\varepsilon_1 = \varepsilon_2 = 0$.  Hence $|Q(B)| = 29 - 10 - 10 = 9$.  Moreover, Claim \ref{claim:var} gives $s = d_B(x_1) + d_B(x_2) \ge 6 + 6 = 12$.

By Lemma~\ref{lem:two-exceptional}  $B$ contains a cycle with at least $31$ chords, a contradiction.  This completes the proof.
\end{proof}

\begin{theorem}
\label{thm:f31-upper}
Every graph on $31$ vertices with minimum degree at least $8$ contains a cycle with at least $31$ chords.
\end{theorem}

\begin{proof}
Suppose $G$ is a counterexample.  We may assume $G$ is connected but not $2$-connected.  If every end-block has private size at most $15$, Proposition~\ref{prop:global-small} yields a contradiction.  Hence some end-block $B^*$ has private size at least $16$.  Let $x$ be its unique cut vertex and set $Q(B^*) = V(B^*) \setminus \{x\}$.  By Fact~\ref{fact:private}, every vertex of $Q(B^*)$ has degree at least $8$ inside $B^*$.  Proposition~\ref{prop:large-terminal} then gives a cycle with at least $31$ chords in $B^*$, and hence in $G$.
\end{proof}

\begin{proposition}
\label{prop:lower}
There exists a graph on $31$ vertices with minimum degree $7$ and no cycle with at least $31$ chords.
\end{proposition}

\begin{proof}
Let $G$ be obtained from two copies of $K_8$ and two copies of $K_9$ by identifying a chosen vertex from each clique into a single common vertex $z$.  Then, $|V(G)| = 1 + 2(8-1) + 2(9-1) = 31$ and
the minimum degree of $G$ is $7$.  

Since the four clique blocks meet pairwise only at the cut vertex $z$,  a cycle cannot use vertices from two different clique blocks. Hence, every cycle of $G$ lies in a single copy of $K_8$ or $K_9$.
A cycle in $K_8$ has length at most $8$, and hence at most $\binom{8}{2} - 8 = 20$ chords.  Similarly, a cycle in $K_9$ has length at most $9$, and hence at most $\binom{9}{2} - 9 = 27$ chords.  Thus no cycle in $G$ has $31$ chords.
\end{proof}

\begin{proof}[\bf Proof of Theorem \ref{thm:KaraKral-conj}.]
Theorem~\ref{thm:f31-upper} gives $f(31,31)\le8$, and Proposition~\ref{prop:lower} gives $f(31,31)\ge8$.
\end{proof}

\section{Concluding remarks}
To conclude this paper, we consider two related problems.
\subsection{Voss's problem and the theorems of Gupta-Kahn-Robertson and Tian-Zang}

In this subsection, we prove Proposition~\ref{prop:four-chord-main} by combining standard properties of $4$-critical graphs with the chord bounds of Gupta-Kahn-Robertson and Tian--Zang.

We recall a standard reduction and basic properties of critical graphs.

\begin{lemma}\label{lem:basic}
If $\chi(G)\ge4$, then $G$ contains a $4$-critical subgraph $H$.  
Moreover, every $4$-critical graph is $2$-connected and has minimum degree at least $3$.
\end{lemma}

\begin{lemma}\label{lem:size}
If $G$ is a $K_4$-free 4-critical graph, then $|V(G)|\ge 6$.
\end{lemma}

\begin{proof}
By Lemma~\ref{lem:basic}, $\delta(G)\ge 3$.  If $|V(G)|\le 3$, then $G$ is 3-colorable.  If $|V(G)|=4$, then a 4-chromatic graph on four vertices is $K_4$, a contradiction.

Assume $|V(G)|=5$.  Since $\delta(G)\ge 3$, the complement $\overline{G}$ has maximum degree at most one, so it is a matching together with isolated vertices.  If $\overline{G}$ has no edge, then $G=K_5$, which contains $K_4$.  If $\overline {G}$ has exactly one edge $ab$, then $V(G)\setminus\{a\}$ spans a $K_4$ in $G$.  Because $G$ is $K_4$-free, $\overline {G}$ has exactly two edges, say $ab$ and $cd$, with fifth vertex $e$.  Coloring $a,b$ with color 1, $c,d$ with color 2, and $e$ with color 3 gives a proper 3-coloring of $G$, a contradiction.  Hence $|V(G)|\ne 5$.
\end{proof}

\begin{lemma}\label{lem:six-vertex-critical}
Every $K_4$-free $4$-critical graph on six vertices is the wheel obtained from a $5$-cycle by adding a vertex adjacent to every vertex of the cycle.
Consequently, it contains a cycle with four chords.
\end{lemma}

\begin{proof}
Let $G$ be such a graph.
Since $\delta(G)\ge3$, we have $\Delta(\overline{G})\le2$.
Let $\alpha(G)$ and $\nu(G)$ denote the independence number and matching
number of $G$, respectively.
Moreover, $\alpha(\overline{G})\le3$ and $\nu(\overline{G})\le2$: an independent set of four vertices in $\overline{G}$ would induce a $K_4$ in $G$, and a matching of size $3$ in $\overline{G}$ would partition $V(G)$ into three independent pairs, yielding a $3$-coloring of $G$.

We first show that $\overline{G}$ is triangle-free.  Suppose $T$ is a triangle in $\overline{G}$.  Because $\Delta(\overline{G})\le2$, $T$ would be a component of $\overline{G}$.  The remaining three vertices cannot induce a triangle in $G$ (otherwise a $K_4$ would appear), so they are $2$-colorable.  Assigning a third color to the vertices of $T$ would then give a $3$-coloring of $G$, a contradiction.

Thus, every component of $\overline{G}$ is a path or a cycle of length at least $4$.  If $\overline{G}$ contained no odd cycle, all its components would be paths or even cycles.  For every such component $D$ we have $\alpha(D)+\nu(D)=|V(D)|$.  Summing over all components gives $\alpha(\overline{G})+\nu(\overline{G})=6$, contradicting $\alpha(\overline{G})\le3$ and $\nu(\overline{G})\le2$.  Hence $\overline{G}$ contains an odd cycle.  Since $\overline{G}$ is triangle-free and has six vertices, this cycle must be a $5$-cycle.  Its vertices already have degree $2$ in $\overline{G}$, so the remaining vertex is isolated.  Therefore, $\overline{G}\cong C_5\cup K_1$, and $G$ is the wheel with rim $C_5$.

The wheel has $10$ edges.  A Hamilton cycle is obtained by joining the universal vertex to the endpoints of a Hamilton path on the rim; it contains $6$ edges, so the remaining $4$ edges are chords.  Thus, $G$ contains a cycle with four chords.
\end{proof}

\begin{proof}[\bf Proof of Proposition~\ref{prop:four-chord-main}]
Let $G$ be a $K_4$-free graph with $\chi(G)\ge4$.  By Lemma~\ref{lem:basic}, $G$ contains a $4$-critical subgraph $H$; clearly $H$ is also $K_4$-free.  Lemmas~\ref{lem:basic} and~\ref{lem:size} imply that $H$ is $2$-connected and satisfies
$\delta(H)\ge3 $ and $|V(H)|\ge6.$

If $\delta(H)\ge4$, then Theorem~\ref{thm:GuptaKahnRobertson} yields a cycle in $H$ with at least
$(\delta(H)+1)(\delta(H)-2)/{2}\ge5$
chords.

It remains to consider $\delta(H)=3$.  If $|V(H)|=6$, Lemma~\ref{lem:six-vertex-critical} gives a cycle with four chords.  Hence we may assume $|V(H)|\ge7$.  Theorem~\ref{thm:tian-zang} then applies unless $H\cong K_{3,m}$ for some $m\ge3$ or $H$ is the Petersen graph.  The former is bipartite and the latter is $3$-colorable, so neither is $4$-critical. 
By Theorem~\ref{thm:tian-zang}, $H$ contains a cycle with at least $4$ chords, and this cycle is also a cycle of $G$ since $H\subseteq G$.
\end{proof}

\subsection{A simple counterexample}

In this subsection we construct an infinite family of $3$-connected graphs where a prescribed pair of edges need not lie on a cycle with at least two chords.

\begin{conjecture}[Voss, \rm{\cite[Conjecture~8.5.6, pp. 169]{VossBook}}]
If $G$ is 3-connected, then every pair of edges of $G$ is contained in a cycle with at least two chords.
\end{conjecture}
We disprove this conjecture by constructing, for every $k\ge2$, a 3-connected graph of order $k+3$.

\begin{proposition}
For every integer $k\ge 2$ there exists a $3$-connected graph $G_k$ containing $k+3$ vertices such that a pair of edges that lies in no cycle with at least two chords.
\end{proposition}

\begin{proof}
Fix $k\ge 2$ and construct $G_k$ as follows. Let
$V(G_k)=\{h,a,c\}\cup\{x_1,x_2,\dots,x_k\}.$
Add the edges $ha$ and $hc$, and for each $i\in\{1,\dots,k\}$ add the three edges $hx_i, ax_i, cx_i$.  There is no edge $ac$ and no edge between any two vertices $x_i,x_j$.  Equivalently,
$G_k = K_{3,k} + \{ha,hc\},$
where the part of size $3$ is $\{h,a,c\}$.

We first verify that $G_k$ is $3$-connected.  Let $S\subseteq V(G_k)$ with $|S|\le 2$ and consider $G_k-S$.  If $h\notin S$, then $h$ remains and is adjacent to every other remaining vertex, so $G_k-S$ is connected.  If $h\in S$, then because $|S|\le 2$ at least one vertex of $\{a,c\}$ and at least one $x_i$ remain.  Every remaining $x_j$ is adjacent to every remaining vertex in $\{a,c\}$, hence the remaining vertices induce a connected subgraph. Thus, deleting at most two vertices cannot disconnect $G_k$, so $G_k$ is $3$-connected.

Now we show that the pair $e=ha,\; f=hc$ lies in no cycle with at least two chords.  Consider $G_k-h$, which is the complete bipartite graph $K_{2,k}$ with parts $\{a,c\}$ and $\{x_1,\dots,x_k\}$.  Let $C$ be a cycle containing both $e$ and $f$.  At vertex $h$ the cycle uses exactly the two incident edges $ha$ and $hc$; deleting $h$ yields a simple $a$–$c$ path in $G_k-h$.

In $G_k-h$ there is no edge $ac$ and no edge between any two $x_i$; therefore the only simple $a$–$c$ paths are of the form $a x_i c$ for some $i$.  Consequently, every cycle containing both $ha$ and $hc$ has the form
$h a x_i c h .$

The only possible chords of this cycle are $hx_i$ and $ac$.  The edge $hx_i$ belongs to $G_k$ while $ac$ does not, so the cycle has exactly one chord.  Hence $\{ha,hc\}$ is not contained in any cycle with at least two chords.

Since $|V(G_k)| = k+3$, the graphs have distinct orders for different $k$; they are pairwise nonisomorphic and thus provide infinitely many examples.
\end{proof}

\section*{Acknowledgment}
The second author is grateful to Jie Ma for providing the reference \cite{APSSV2026} and for introducing him to the Erd\H{o}s Problems website. The authors also thank Thomas Bloom for creating and sustaining the website \cite{Bloom1091}. 

\section*{Declaration of AI usage} An AI assistant (Chatgpt 5.5) was used in the search for a construction in Theorem \ref{thm:ghm-false} with the initial requirement based on the Haj\'os join suggested by the second author. After several rounds of refinement, talks with AI and checking, we have the current construction. The AI
assistant also helped with proofreading, grammar checking, and language polishing. Except for the work related to Theorem \ref{thm:ghm-false}, all proofs in this paper and their
presentation were carried out by the human authors, including the key ideas of extending the construction of Alexeev et al. and reasoning on the technique of Ash's theorem. The present authors bear full responsibility for the correctness of the proofs and the rigor of the writing in this paper.

\end{document}